\documentclass[11pt]{article}
\usepackage{amsmath,amsthm,amsfonts,amssymb,bm,wasysym}
\usepackage{epsfig}
\usepackage[usenames]{color}
\usepackage{verbatim}
\usepackage{hyperref}
\usepackage{multicol}
\usepackage{comment}
\usepackage{float}
\usepackage{graphicx}
\parskip \medskipamount
\numberwithin{equation}{section}
\newtheorem{theorem}{Theorem}[section]
\newtheorem{lemma}[theorem]{Lemma}

\newtheorem{corollary}[theorem]{Corollary}
\theoremstyle{definition}
\newtheorem{definition}[theorem]{Definition}
\newtheorem{remark}[theorem]{Remark}

\newtheorem{conjecture}[theorem]{Conjecture}
\numberwithin{equation}{section}

\def\N{\mathbb{N}}

\def\R{\mathbb{R}}

\renewcommand{\phi}{\varphi}
\renewcommand{\epsilon}{\varepsilon}

\newcommand{\Prp}[2]{\mathbb{P}_{#1}\left(#2\right)} % \Prp{a}{b} yields P_a(b); a is a subscript, e.g. the initial state in a Markov Process
\newcommand{\Exp}[2]{\mathbb{E}_{#1}\left[#2\right]} % \Exp{a}{b} yields E_a[b]; a is a subscript, e.g. the initial state in a Markov Process
\newcommand{\Prpwo}[1]{\mathbb{P}_{#1}} % \Prpwo{a}{b} yields P_a; a is a subscript

\allowdisplaybreaks

\newcommand{\1}{{\text{\Large $\mathfrak 1$}}}

\newcommand{\thit}{t_{\mathrm{hit}}}

\newcommand{\trel}{t_{\mathrm{rel}}}

\newcommand{\pr}[1]{\mathbb{P}\!\left(#1\right)}
\newcommand{\E}[1]{\mathbb{E}\!\left[#1\right]}
\newcommand{\estart}[2]{\mathbb{E}_{#2}\!\left[#1\right]}
\newcommand{\prstart}[2]{\mathbb{P}_{#2}\!\left(#1\right)}
\newcommand{\prcond}[3]{\mathbb{P}_{#3}\!\left(#1\;\middle\vert\;#2\right)}

\newcommand{\escond}[3]{\mathbb{E}_{#3}\!\left[#1\;\middle\vert\;#2\right]}

\newcommand{\tn}{|\kern-.1em|\kern-0.1em|}

\def\Mbad{M^{\rm bad}}
\def\Mgood{M^{\rm good}}

\newcommand\be{\begin{equation}}
\newcommand\ee{\end{equation}}

\def\eps{\varepsilon}

\newcommand{\thmref}[1]{Theorem~\ref{thm:#1}} % Theorem tag equals ``thm''
\newcommand{\lemref}[1]{Lemma~\ref{lem:#1}} % Lemma tag equals ``lem''
\newcommand{\corref}[1]{Corollary~\ref{cor:#1}} % Corollary = ``cor''
\newcommand{\eqnref}[1]{(\ref{eq:#1})} % Equation = ``eq''
\newcommand{\ignore}[1]{}

\begin{document}
\title{\bf Random walks colliding before getting trapped}

\author{
Louigi Addario-Berry\thanks{McGill University, Montr\'eal, Canada; louigi.addario@mcgill.ca} \and Roberto I. Oliveira\thanks{IMPA, Rio de Janeiro, Brazil; rimfo@impa.br. Supported by a {\em Bolsa de Produtividade em Pesquisa} from CNPq.} \and
Yuval Peres\thanks{Microsoft Research, Redmond, Washington, USA; peres@microsoft.com} \and Perla Sousi\thanks{University of Cambridge, Cambridge, UK;   p.sousi@statslab.cam.ac.uk}
}
\date{}
\maketitle
\begin{abstract}
Let $P$ be the transition matrix of a finite, irreducible and reversible Markov chain. We say the continuous time Markov chain $X$ has transition matrix $P$ and speed~$\lambda$ if it jumps at rate~$\lambda$ according to the matrix $P$. Fix $\lambda_X,\lambda_Y,\lambda_Z\geq 0$, then let $X,Y$ and $Z$ be independent Markov chains with transition matrix $P$ and speeds $\lambda_X,\lambda_Y$ and $\lambda_Z$ respectively, all started from the stationary distribution. What is the chance that $X$ and $Y$ meet before either of them collides with $Z$? 
For each choice of $\lambda_X,\lambda_Y$ and $\lambda_Z$ with $\max(\lambda_X,\lambda_Y)>0$, we prove a lower bound for this probability which is uniform over all transitive, irreducible and reversible chains. In the case that $\lambda_X=\lambda_Y=1$ and $\lambda_Z=0$ we prove a strengthening of our main theorem using a martingale argument. We provide an example showing the transitivity assumption cannot be removed for general $\lambda_X,\lambda_Y$ and $\lambda_Z$.
\newline
MSC 2010 \emph{subject classifications.} Primary 60J10.
\end{abstract}

\section{Introduction}\label{sec:intro}

Consider three independent random walks $X,Y,Z$ over the same finite connected graph. What is the probability that $X,Y$ meet at the same vertex before either of them meets $Z$? If the initial distributions of the three walkers are the same, this probability is at least $1/3$ by symmetry, at least if we assume that ties (i.e.~triple meetings) are broken symmetrically. 

Now consider a similar problem where the initial states $X_0,Y_0,Z_0$ are all sampled independently from the same distribution, but $Z$ stays put while $X$ and $Y$ move. What is the probability that $X$ and $Y$ meet before hitting $Z$? 

There are several examples of bounds~\cite{AldFill, Oliveira, SousiWinkler} relating the meeting time of two random walks to the hitting time of a fixed vertex by a single random walk. These typically provide upper bounds for meeting times in terms of worst-case or average hitting times, sometimes up to constant factors. In light of this,
it seems natural to conjecture that the probability in question is at least $1/3$. However, the previous argument by symmetry fails. In fact, to the best of our knowledge, no known universal lower bound for this probability is known.

It will be convenient to consider the problem in continuous time. 
For the remainder of the paper let $P$ be the transition matrix of an irreducible and reversible Markov chain on a finite state space with stationary distribution $\pi$. 
Let $X$ and $Y$ be two independent continuous time Markov chains that jump at rate~$1$ according to the transition matrix $P$ and let $Z\sim \pi$ be independent of~$X$ and $Y$. 

We define $M^{X,Y}$ to be the first time $X$ and $Y$ meet, i.e.
\[
M^{X,Y} = \inf\{t\geq 0: \,X_t=Y_t\}.
\]
We also define:
\[M^{W,Z}=\inf\{t\geq 0:\, X_t=Z\}\,\,(W\in\{X,Y\}).\] We write $\Mgood=M^{X,Y}$ and $\Mbad=M^{X,Z}\wedge M^{Y,Z}$. 

\subsection{Main results}

Our first result proves a universal lower bound on the probability $\pr{\Mgood<\Mbad}$ for the class of transitive chains. First we recall the definition.
\begin{definition}\rm{
Fix a chain with transition matrix $P$ and state space $E$. An {\em automorphism} of $P$ is a bijection $\phi:E\to E$ such that $P(z,w) = P(\phi(z),\phi(w))$ for all $z,w$. The chain $P$ is {\em transitive} if for all $x,y \in E$ there exists an automorphism $\phi$ of $P$ with $\phi(x)=y$.
}
\end{definition}

\begin{theorem}\label{thm:put}
Let $P$ be the transition matrix of a finite irreducible and reversible chain with two or more states. Assume $X_0$ and $Y_0$ are independent with law $\pi$. If $P$ is transitive, then
\[
\pr{\Mgood<\Mbad}\geq \frac{1}{4}.
\]
\end{theorem}

Next we consider a more general setup. We say that a random walk $W$ has speed $\lambda_W$ and transition matrix $P$, if it jumps at rate $\lambda_W$ according to the matrix $P$.

Suppose again that $P$ is the transition matrix of an irreducible and reversible Markov chain on a finite state space with stationary distribution $\pi$. Let~$\lambda_X=1, 0\leq \lambda_Y\leq 1$ and $0\leq \lambda_Z<\infty$. Let $X,Y$ and $Z$ be three independent continuous time Markov chains with speeds $\lambda_X, \lambda_Y$ and $\lambda_Z$ respectively and transition matrix~$P$. 

For the remainder of the paper, we write $\mathbb{P}$ for the probability measure under which $X_0, Y_0$ and~$Z_0$ are independent with law $\pi$. We also write $\mathbb{P}_{a,b,c}$ in the case when $(X_0, Y_0, Z_0)=(a,b,c)$. For computations that only involve 
two chains we drop one index writing only $\mathbb{P}_{a,b}$; which two chains are involved will always be clear from context. Likewise, we write $\mathbb{P}_a$ when only one chain is involved. We define $M^{X,Y}$ as above and redefine:
\[M^{W,Z}=\inf\{t\geq 0:\, X_t=Z_t\}\,\,(W\in\{X,Y\}).\] $\Mgood=M^{X,Y}$ and $\Mbad=M^{X,Z}\wedge M^{Y,Z}$ are defined as before. Again we are interested in uniform lower bounds on the probability  of the event $\{\Mgood<\Mbad\}$ that have good dependence on the three speeds.

\begin{theorem}\label{thm:redmond}
There exists $c>0$ such that the following holds. 
Let $P$ be the transition matrix of a transitive and reversible chain with stationary distribution $\pi$ and at least two states. Suppose that $X,Y$ and $Z$ are three independent continuous time Markov chains with speeds $\lambda_X=1, \lambda_Y\leq 1$ and $0\leq \lambda_Z<\infty$ and transition matrix $P$ started from $\pi$. Then 
\[
\pr{\Mgood<\Mbad}\geq \frac{c}{(\sqrt{1+\lambda_Z} + \sqrt{\lambda_Y+\lambda_Z})^2}\, .
\]
\end{theorem}

The proof shows that we may take $c=1/4752$, 
which implies a version of Theorem~\ref{thm:put} with $1/4$ replaced by $1/((\sqrt{2}+1)^2\cdot 4752)$. The constant $c$ most likely can be improved, but the dependence of the lower bound on $\lambda_Z$ is sharp when $\lambda_Z\nearrow +\infty$. Indeed, if $P$ is simple random walk over a large complete graph with $n$ vertices, then
$$\pr{\Mgood\leq \Mbad}= \frac{1+\lambda_Y}{2(1+\lambda_Y+\lambda_Z)} - O\left({\frac{1}{n}}\right),$$
where the term $O(1/n)$ corresponds to the possibility of meetings at time $0$. 

It is natural to ask if our theorems can be extended to all transitive chains. The next theorem shows that the answer is no for the more general Theorem \ref{thm:redmond}. The theorem essentially asserts that there are graphs where typical meeting times are much smaller than typical hitting times.
\begin{theorem}\label{thm:counter}
For all $\epsilon>0$ there exists a finite connected graph $G$ such that if $P$ corresponds to simple random walk on $G$ and $\lambda_X=1, \lambda_Y=0$ and $\lambda_Z=1$, then
$\pr{\Mgood\leq \Mbad}<\epsilon$.
\end{theorem}
On the other hand, we believe that for certain values of $\lambda_X,\lambda_Y$ and $\lambda_Z$,  universal lower bounds are possible without transitivity. Here is a concrete conjecture, which relates to the setting of Theorem~\ref{thm:put}.

\begin{conjecture}\label{conj:1}If $\lambda_Y=\lambda_X=1$ and $\lambda_Z=0$, the inequality
$$\pr{\Mgood\leq \Mbad}\geq 1/3$$
holds for all finite irreducible and reversible chains $P$.\end{conjecture}

Alexander Holroyd (personal communication) pointed out an example 
showing that for any $\delta>0$ there exist transitive chains for which $\pr{\Mgood \leq \Mbad}\leq 1/3+\delta$. We describe this example in Section~\ref{sec:sharpness}. This means that, if true, Conjecture \ref{conj:1} is best possible even for transitive chains. However, we note that any uniform lower bound 
$$\pr{\Mgood\leq \Mbad} \geq c>0$$
for all $P$, and for $\lambda_X, \lambda_Y$ and $\lambda_Z$ as in Conjecture~\ref{conj:1}, would be a new result.

\begin{remark}\rm{
Without reversibility, the conjecture fails badly. Consider a clockwise continuous time random walk on a cycle of length $2n$. More precisely, with $P= (p_{ij})_{1\leq i,j\leq n}$ we have $p_{ij}=1$ if $j=(i+1)\bmod n$ and $p_{ij}=0$ otherwise. The distance between independent random walkers behaves as continuous time simple symmetric random walk reflected at $0$ and $n$. 
So started from stationarity, it typically takes such walkers time of order $n^2$ to meet. On the other hand, the hitting time of any point is at most of order $n$.
}
\end{remark}

Before we continue, we say a few words about the main proof ideas. The unifying theme of the proofs of Theorems~\ref{thm:put} and~\ref{thm:redmond} is the relationship between meeting times and hitting times of single vertices when $P$ is transitive. Aldous and Fill \cite[Chapter~14/Proposition~5 and Chapter~3]{AldFill} have related the expected values of these random variables via martingales, and we use similar ideas to prove Theorem~\ref{thm:put}. 

For the proof of Theorem~\ref{thm:redmond}, we need a stronger result establishing {\em identities in distribution} of meeting and hitting times, which (somewhat surprisingly) seems to be new: see \lemref{identity} below. The proof of Theorem~\ref{thm:redmond} requires several other tools, including small time estimates for hitting times (\lemref{smalltimetransitive}) and an occupation time formula for product chains (\lemref{stop}).

The proof of Theorem~\ref{thm:counter} builds a graph with two parts: the ``Up"~part concentrates the bulk of the stationary measure, but the ``Down"~ part is where meetings tend to happen, and they happen fast. As a result, only a negligible fraction of the ``Up"~part is explored before $X$ and $Z$ meet, and the upshot is that $M^{X,Y}>M^{X,Z}$ with high probability.

\section{The 1/4 lower bound}\label{sec:perla}

In this section we prove Theorem~\ref{thm:put}. The argument is fairly short, and much simpler than the one for the more general Theorem~\ref{thm:redmond}. Before presenting the proof, we recall  some standard facts about hitting times which are also used later on. 

The hitting time of a state $z\in\Omega$ by $X$ is the first time $t$ at which $X_t=z$, i.e.
\begin{equation}\label{eq:defhit}\tau^X_z:=\inf\{t\geq 0\,:\,X_t=z\}.
\end{equation}
We define $\tau^Y_z$ similarly and we also let
\begin{equation}\label{eq:defthit}t^{*}_{\rm hit}:= \max_{x\in\Omega}\,\estart{\tau^X_x}{\pi} \quad \text{and} \quad t_{\rm hit}:= \max_{(x,z)\in\Omega}\,\estart{\tau^X_z}{x}.
\end{equation}

Whenever there is no confusion, i.e.\ if there is a single chain in question, we will drop the dependence on $X$ or $Y$ from the notation of the hitting times.

\begin{lemma}\label{lem:knownfacts}
For any reversible chain with two or more states we have 
\[
0<t_{\rm hit}\leq 2t^{*}_{\rm hit}.
\]
Moreover, if $X$ is a reversible and transitive chain, then for all $x,z\in \Omega$ and all $t\geq 0$ we have 
\[
\prstart{\tau_z\leq t}{x} = \prstart{\tau_x\leq t}{z}.
\]
\end{lemma}
\begin{proof}[\bf Proof]

For a proof of the first assertion in discrete time we refer the reader to~\cite[Lemma~10.2]{LevPerWil}. A proof of the symmetry property specific to transitive chains can be found in~\cite[Lemma~1, Chapter~7]{AldFill}.
\end{proof}

\begin{lemma}[Aldous]\label{lem:AF}
Let $P$ be an irreducible and reversible transition matrix.
Suppose that~$X$ and~$Y$ are independent continuous time Markov chains that jump at rate~$1$ according to the transition matrix $P$. For all $x,z\in \Omega$ define $f(x,z):=\estart{\tau^X_z}{x}$. 
Then $f(X_t,z)+t$, $f(Y_t,z)+t$ and $f(X_t,Y_t)+2t$ are martingales up to time $$S:=\Mgood\wedge \Mbad = M^{X,Y}\wedge \tau^X_z\wedge \tau^Y_z,$$ for any initial states $(x,y)\in\Omega^2$ and any $z\in\Omega$. \end{lemma}

For a proof of the lemma above we refer the reader to~\cite[Chapter~14/Proposition~5 and Chapter~3]{AldFill}.

\begin{proof}[\bf Proof of Theorem~\ref{thm:put}] 

Since $X$ and $ Y$ are two independent copies of the same chain, we have $\estart{\tau^X_b}{a}=\estart{\tau^Y_b}{a}$ for all $a,b$.
By Lemma~\ref{lem:AF} we now get that $(G_t)_{t\geq 0}$ is a martingale up to time~$S$, where
\[
G_t := \estart{\tau^X_z}{X_t} + \estart{\tau^X_z}{Y_t} - \estart{\tau^X_{Y_t}}{X_t}\,\,\,(t\geq 0).
\]
This martingale is bounded (because the state space is finite). The fact that the chain is finite and irreducible implies $S<\infty$ almost surely for all initial states. We deduce from optional stopping that 
$$\E{G_0} = \E{G_S}.$$
The left hand side above is given by the quantity $t^{*}_{\rm hit}$ defined in \eqref{eq:defthit}. This is because  
\begin{equation}\label{eq:transitivitycomesinhere}\E{G_0}=\estart{\tau^X_z}{\pi} + \estart{\tau^X_z}{\pi} - \sum_{y\in\Omega}\pi(y)\estart{\tau^X_{y}}{\pi} = \Exp{\pi}{\tau^X_z},\end{equation}
where the second equality follows from the fact that for a transitive chain, $\estart{\tau_y^X}{\pi}$ is independent of $y$. Using this a second time yields
\begin{equation}\label{eq:equalitythitstar}t^{*}_{\rm hit} = \E{G_S}.\end{equation}

On the other hand, at time $S$ we have two alternatives.
\begin{itemize}
\item If $\tau^X_z\wedge \tau^Y_z\leq \Mgood$, either $X_S=z$, and then $G_S = \Exp{Y_S}{\tau^X_z} - \Exp{z}{\tau^X_{Y_s}}$, or $Y_S=z$, in which case
$G_S = \Exp{X_S}{\tau^X_{z}}-\Exp{X_S}{\tau^X_{z}}$. In both cases $G_S=0$: this is obviously true in the second case, and follows from \lemref{knownfacts} in the first case.
\item On the other hand, if $\Mgood<\tau^X_z\wedge \tau^Y_z$, then $G_S = 2\Exp{X_S}{\tau^X_z}\leq 2t_{\rm hit}$. 
\end{itemize}
We deduce that $$G_S\leq 2t_{\rm hit}\,\1(\Mgood<\tau^X_z\wedge \tau^Y_z).$$ Plugging this into \eqref{eq:equalitythitstar} gives
$$t^{*}_{\rm hit}\leq 2\,t_{\rm hit}\,\pr{\Mgood<\Mbad}.$$
Using that $t_{\rm hit}\leq 2t^{*}_{\rm hit}$ from Lemma~\ref{lem:knownfacts} finishes the proof.
\end{proof}

\section{Towards the general lower bound}\label{sec:towards}

In this section we collect the tools that we will use in the proof of~\thmref{redmond}. We first argue that the obvious ``fix"~to the proof of~\thmref{put} does not work in all cases. Indeed, a straightforward extension of Lemma~\ref{lem:AF} establishes that
$$G_t:=f(X_t,Z_t) + f(Y_t,Z_t) - \frac{1+\lambda_Y+2\lambda_Z}{1+\lambda_Y}\,f(X_t,Y_t)$$
is a martingale up to time $S$. One can see that in this case
$$\E{G_0} = \left(1 - \frac{2\lambda_Z}{1+\lambda_Y}\right)\,t^{*}_{\rm hit},$$
which easily yields
\[
\pr{\Mgood<\Mbad}\geq \frac{1}{4}\cdot \left(1-\frac{2\lambda_Z}{1+\lambda_Y} \right).
\]
In particular, we obtain the same bound as in Theorem~\ref{thm:put} provided that $\lambda_Z=0$ and $\lambda_X\neq 0$. However, this bound becomes useless when $\lambda_Z>(1+\lambda_Y)/2$. Other linear combinations of $f(X_t,Z_t)$, $f(Y_t,Z_t)$ and $f(X_t,Y_t)$ also fail to achieve our goal when $\lambda_Z$ is large. So, in general a different strategy is needed.

\subsection{Hitting times for states and trajectories}\label{sec:firsthit}

In this section we collect some results on hitting times for a single Markov chain. 

Recall the quantity $t^{*}_{\rm hit}$ defined in \eqnref{defthit}. The next lemma shows that, up to a constant factor, $t^{*}_{\rm hit}$ also bounds expected hitting times of moving targets, from arbitrary initial states. This lemma is essentially due to Oliveira~\cite[Lemma~1.1]{Oliveira}, but in this particular form it appeared in~\cite{SousiWinkler}. 
\begin{lemma}\label{lem:meetingtime}
Let $X$ be a reversible Markov chain taking values in $\Omega$ and ${\bf h}=(h_t)_{t\geq 0}$ a deterministic, c\`adl\`ag, $\Omega$-valued trajectory. If 
$$\tau_{\bf h}:=\inf\{t\geq 0\,:\,X_t = h_t\},$$ then for any $x\in\Omega$, 
$$\Exp{x}{\tau_{\bf h}}\leq 11\,t^{*}_{\rm hit}.$$\end{lemma}
\begin{proof}[\bf Proof]
In~\cite{SousiWinkler} using~\cite[Lemma~1.1]{Oliveira} it is
proved that 
$$\Exp{x}{\tau_{\bf h}}\leq c\,t_{\rm hit}.$$
for a universal constant $c>0$, where $t_{\rm hit}$ is as in \eqnref{defthit}. Inspection of the proof \cite{SousiWinkler} shows that $c\leq 4 + 5/4$, therefore $2c\leq 11$. Lemma~\ref{lem:knownfacts} finishes the proof.\end{proof}

For a reversible transition matrix $P$ we let  $\lambda_*=\max_{i\geq 2}|\lambda_i|$, where  $1=\lambda_1\geq \lambda_2\geq \lambda_3,\ldots$ are its eigenvalues in decreasing order. We define the relaxation time via $\trel=(1-\lambda_*)^{-1}$.

The next lemma bounds the probability that $\tau_z$ is small for reversible chains. Clearly, if no assumption is made on $z$ and on the starting state $x$, such an estimate cannot be very good (think of two adjacent points on a path). The next lemma shows that if we choose the ``right"~starting state, and only consider the majority of possible target states, we can show that $\tau_z$ dominates an exponential random variable with mean $t_{\rm rel}$. This will be used later to upper bound the probability that $\tau_z/t_{\rm hit}$ is very small.

\begin{lemma}\label{lem:smalltime1} 

Let $X$ be a reversible chain. There exist $x\in \Omega$ and a subset $A\subset \Omega$ with stationary measure $\pi(A)\geq 1/2$ such that, if $\tau_A:=\min_{z\in A}\tau_z$, then for any $t>0$,
$$\Prp{x}{\tau_A>t}\geq e^{-\frac{t}{t_{\rm rel}}}.$$\end{lemma}

\begin{proof}[\bf Proof]
We use the spectral theory of reversible chains \cite[Section 3.4]{AldFill}. The first step is to note that $P$ has a non-zero eigenfunction~$\phi:\Omega\to\R$ such that $$P\phi = \left(1 - \frac{1}{t_{\rm rel}}\right)\,\phi.$$ This eigenfunction is orthogonal to the constant eigenfunction in the inner product induced by $\pi$, so it must take both positive and negative values. We may assume without loss of generality that the set 
$$A:=\{z\in\Omega\,:\,\phi(z)\leq 0\}$$
has measure $\pi(A)\geq 1/2$ (if that is not the case, replace $\phi$ with $-\phi$).
Choose $x\in\Omega$ with $\phi(x)>0$ as large as possible. On one hand
\begin{equation}\label{eq:eigen}\forall t>0,\, u\in\Omega\,:\,\Exp{u}{\phi(X_t)} = [e^{t(P-I)}\,\phi](u) = e^{-\frac{t}{t_{\rm rel}}}\,\phi(u).
\end{equation}
In particular, for all $u\in A$ we have $\estart{\phi(X_t)}{u}\leq 0$. Since $X_{\tau_A}\in A$, the strong Markov property then gives 
\begin{eqnarray*}
\Exp{x}{\phi(X_t)\,\1(\tau_A\leq t)} &=& \Exp{x}{\Exp{X_{\tau_A}}{\phi(X_{t-\tau_A})}\1(\tau_A\leq t)}\\  &\leq &\Exp{x}{0\cdot \1(\tau_A\leq t)}=0.
\end{eqnarray*}
Plugging this into \eqnref{eigen} with the choice $u=x$, and recalling $\phi(X_t)\leq \phi(x)$ always, we obtain
$$e^{-\frac{t}{t_{\rm rel}}}\phi(x) = \Exp{x}{\phi(X_t)}\leq \Exp{x}{\phi(X_t)\,\1(\tau_A>t)}\leq \phi(x)\,\Prp{x}{\tau_A>t}.$$
Dividing both sides by $\phi(x)$ (which is $>0$) finishes the proof.\end{proof}

\subsection{Results for transitive chains}

In this section we prove results on hitting times under the assumption that $P$ is transitive and reversible.

\subsubsection{A small-time estimate for hitting times}

We start by recalling a result which follows from the complete positivity of the law of $\tau_z$ when starting from $\pi$ together with bounds from Aldous and Brown~\cite{AldBrown}.

\begin{lemma}[Aldous-Brown]\label{lem:AldBrown}
Let $P$ be an irreducible and reversible chain. Define $f(s) = \escond{\tau_z-s}{\tau_z>s}{\pi}$ for all $s>0$. Then $f$ is an increasing function and 
\[
\sup_s f(s) \leq \estart{\tau_z}{\pi} + \trel.
\]
\end{lemma}

\begin{proof}[\bf Proof]

Let $Q=I-P$ and $Q_z$ be the restriction of $Q$ to $\Omega\setminus\{z\}$. We recall the complete positivity of the law of $\tau_z$ starting from $\pi$ (see for instance~\cite[eqn. (18)]{AldBrown}): there exist non-negative constants $(p_i, 1\leq i\leq n)$ such that for all $t$
\begin{align*}
\prstart{\tau_z>t}{\pi} = \sum_{i=1}^{m} p_i e^{-\gamma_i t},
\end{align*}
where $0<\gamma_1<\ldots<\gamma_m$ are the distinct eigenvalues of $-Q_z$. Note that $\gamma_1^{-1}=\estart{\tau_z}{\alpha}$, where $\alpha$ is any quasistationary distribution on $\Omega\setminus\{z\}$ corresponding to the eigenvalue $\gamma_1$.

Using the above representation we can rewrite $f$ as follows
\begin{align*}
f(s) = \int_{s}^{\infty} \frac{\prstart{\tau_z>t}{\pi}}{\prstart{\tau_z>s}{\pi}}\,dt =\frac{\sum_{i=1}^{m}p_i e^{-\gamma_i s}/\gamma_i}{\sum_{i=1}^{m}p_i e^{-\gamma_i s}}.
\end{align*}
A straightforward differentiation now gives that $f$ is increasing. From the above expression we also deduce
\begin{align}\label{eq:limitf}
\lim_{s\to\infty} f(s) =\frac{1}{\gamma_1} = \estart{\tau_z}{\alpha}.
\end{align}
From~\cite[Corollary~4]{AldBrown} we have
\[
\estart{\tau_z}{\alpha}\leq \estart{\tau_z}{\pi} + \trel.
\]
Therefore, using this and the fact that $f$ is increasing and~\eqref{eq:limitf} we conclude that for all $s$
\[
f(s) \leq \estart{\tau_z}{\alpha}\leq \estart{\tau_z}{\pi}+ \trel
\]
which completes the proof.
\end{proof}

The next lemma essentially improves upon \lemref{smalltime1} from Section~\ref{sec:firsthit}.

\begin{lemma}\label{lem:smalltimetransitive}
Suppose that $P$ is reversible and transitive. Then for any $x\in\Omega$, there exists a subset~$A_x\subset \Omega$ with $\pi(A_x)\geq 1/2$ such that, for any $\theta>0$,
$$\frac{1}{\pi(A_x)}\sum_{z\in A_x}\pi(z)\Prp{x}{\tau_z\leq \theta\,t_{\rm hit}} \leq 6\sqrt{\theta}$$\end{lemma}

\begin{proof}[\bf Proof]

By \lemref{smalltime1} there exists $x\in \Omega$ and a set~$A=A_x$ with $\pi(A_x)\geq  1/2$ such that 
$$\Prp{x}{\tau_A>t}\geq e^{-\frac{t}{t_{\rm rel}}}.$$
Since the chain is transitive, this in fact holds for all $x$. We now fix $x\in \Omega$ and $\theta>0$.  

We will consider two cases separately: $t_{\rm rel}< \sqrt{\theta}\,t_{\rm hit}$ and $t_{\rm rel}\geq \sqrt{\theta}\,t_{\rm hit}$.

Suppose first that $t_{\rm rel}\geq  \sqrt{\theta}\,t_{\rm hit}$. Since $\Prp{x}{\tau_z>t}\geq \Prp{x}{\tau_A>t}$ for all $z\in A$, we obtain
$$\frac{1}{\pi(A)}\sum_{z\in A}\pi(z)\,\Prp{x}{\tau_z>\theta\,t_{\rm hit}}\geq e^{-\frac{\theta\,t_{\rm hit}}{t_{\rm rel}}}\geq 1-\frac{\theta\,t_{\rm hit}}{t_{\rm rel}}\geq 1-\sqrt{\theta}$$
which concludes the proof in this case.

Suppose next that $t_{\rm rel}< \sqrt{\theta}\,t_{\rm hit}$. In this case it suffices to prove
\begin{equation}\label{eq:goalcase1}\forall z\in \Omega\,:\,\Prp{\pi}{\tau_z>\theta\,t_{\rm hit}}\geq 1 - 3\sqrt{\theta}.\end{equation}
To see that this suffices, we use the fact that $P$ is transitive and apply Lemma~\ref{lem:knownfacts} to obtain that $\Prp{x}{\tau_z>t}$ is symmetric in $x$ and $z$. As a result, \eqnref{goalcase1} implies 
$$1-3\sqrt{\theta}\leq \Prp{\pi}{\tau_x>\theta\,t_{\rm hit}} = \sum_{z\in \Omega}\pi(z)\Prp{z}{\tau_x>\theta\,t_{\rm hit}}=\sum_{z\in \Omega}\pi(z)\Prp{x}{\tau_z>\theta\,t_{\rm hit}}.$$
Therefore we obtain
\[
\sum_z \pi(z) \Prp{x}{\tau_z\leq \theta\,t_{\rm hit}}\leq 3\sqrt{\theta}.
\]
Since $\pi(A)\geq 1/2$, we conclude
\[
\frac{1}{\pi(A)}\sum_{z\in A}\pi(z) \Prp{x}{\tau_z\leq \theta\,t_{\rm hit}} \leq 6\sqrt{\theta}.
\]
It remains to prove~\eqref{eq:goalcase1}. Since $P$ is transitive, $\Exp{\pi}{\tau_z}=t^{*}_{\rm hit}$ is independent of $z$. Moreover, we are assuming that $t_{\rm rel}\leq \sqrt{\theta}\,t_{\rm hit}$, and hence using Lemma~\ref{lem:knownfacts} we get that $t_{\rm rel}\leq 2\sqrt{\theta}\,t^{*}_{\rm hit}$
so 
\begin{align}\label{eq:upperboundexp}
\Exp{\pi}{\tau_z} + t_{\rm rel}\leq (1+2\sqrt{\theta})\,t^{*}_{\rm hit}.
\end{align}
We now obtain for all $s\geq 0$ and for all $z\in \Omega$
\begin{eqnarray*}
t^{*}_{\rm hit} = \Exp{\pi}{\tau_z}\leq s + \Exp{\pi }{\tau_z-s\mid \tau_z>s}\,\Prp{\pi}{\tau_z>s}
&\leq & s+(\estart{\tau_z}{\pi}+\trel)\prstart{\tau_z>s}{\pi}
\\ &\leq & s + (1+  2\sqrt{\theta})\,t^{*}_{\rm hit}\,\Prp{\pi}{\tau_z>s},\end{eqnarray*}
where the second inequality follows from Lemma~\ref{lem:AldBrown} and the final one from~\eqref{eq:upperboundexp}.
Taking $s=\theta\,t^{*}_{\rm hit}$ gives
\begin{equation*} 
\Prp{\pi}{\tau_z>\theta\,t^{*}_{\rm hit}} \geq \frac{1-\theta}{1+2\sqrt{\theta}} \geq 1-2\sqrt{\theta}.
\end{equation*}
This now finishes the proof of~\eqref{eq:goalcase1}, since for all $z\in \Omega$ we have 
\begin{align*}
\prstart{\tau_z>\theta \thit}{\pi}\geq \prstart{\tau_z> 2 \theta \thit^{*}}{\pi} \geq 1 - 2\sqrt{2}\sqrt{\theta} \geq 1-3\sqrt{\theta},
\end{align*}
where for the first inequality we used again Lemma~\ref{lem:knownfacts}.\end{proof}

\subsubsection{Distributional identities for meeting and hitting times}

Our next result shows that the distributions of hitting and meeting times are intimately related for transitive and reversible $P$.  

\begin{lemma}\label{lem:identity}

Let $P$ be a reversible and transitive transition matrix. Let~$X, Y$ and $Z$ be three independent continuous time Markov chains with speeds $\lambda_X=1, \lambda_Y\geq 0$ and $\lambda_Z\geq 0$ and transition matrix $P$. Then for all $(x,z)\in\Omega^2$, the distribution of $\frac{\tau^X_z}{\lambda_Y + \lambda_Z}$ under $\Prpwo{x}$ is the same as the distribution of $M^{Y,Z}$ under $\Prpwo{(x,z)}$. 
\end{lemma}

A special case of this lemma is when $\lambda_X=\lambda_Z=1$, in which case we obtain the following corollary of \lemref{identity}.
$$\Exp{(x,z)}{M^{X,Z}} = \frac{1}{2}\,\Exp{x}{\tau^X_z}.$$
This equality is well known and is usually proven by martingale methods such as the ones used in the proof of \thmref{put}. Somewhat oddly, it seems that \lemref{identity} is new, or at least was not widely known before. We also note the following corollary of \lemref{smalltimetransitive} and \lemref{identity}.

\begin{corollary}\label{cor:smalltimetransitive}
Let $P$ be a transitive and reversible transition matrix and let $X,Y$ and $Z$ be three independent continuous time Markov chains with speeds $\lambda_X=1$ and $\lambda_Y, \lambda_Z\geq 0$ respectively and transition matrix $P$. Then for all $x\in \Omega$ there exists a subset $A_x\subset\Omega$ with $\pi(A_x)\geq 1/2$ such that:
$$\frac{1}{\pi(A_x)}\,\sum_{z\in A_x}\,\pi(z)\,\Prp{(x,z)}{M^{X,Z}\leq \theta\,t_{\rm hit}}\leq 6\sqrt{(1+\lambda_Z)\,\theta}$$
and
$$\frac{1}{\pi(A_x)}\,\sum_{z\in A_x}\,\pi(z)\,\Prp{(x,z)}{M^{Y,Z}\leq \theta\,t_{\rm hit}}\leq 6\sqrt{(\lambda_Y+\lambda_Z)\,\theta}.$$\end{corollary}

\begin{proof}[\bf Proof of Lemma~\ref{lem:identity}]
Define the functions
$$g_{(x,z)}(t):=\Prp{x}{\tau_z\leq (\lambda_Y+\lambda_Z)\,t}\mbox{ and }f_{x,z}(t):=\Prp{(x,z)}{M^{Y,Z}\leq t}\,\,((x,z)\in\Omega^2,\,t\geq 0).$$
We will be done once we show that $g_{x,z}(t) = f_{(x,z)}(t)$ for all $(x,z)\in\Omega^2$ and $t\geq 0$. These equalities are true (by inspection) when $t=0$. We are going to show that the functions $(f_{(x,z)}(\cdot))_{(x,z)\in\Omega^2}$ and $(g_{x,z}(\cdot))_{(x,z)\in\Omega^2}$ satisfy the same linear system of ordinary differential equations (with the derivatives at $t=0$ interpreted as right derivatives). Then the equality for all $t\geq 0$ will follow from the general uniqueness theory of linear ODE's.

To prove that $f$ and $g$ satisfy the same system of ODE's,  we will use a standard formula for the cumulative distribution function of a hitting time. If $(V_t)_{t\geq 0}$ is an irreducible continuous time Markov chain over a set $\Omega_V$ with transition rates $q(v,w)$, and $A\subset \Omega_V$ is a nonempty subset of the state space, the hitting time $\tau^V_A$ of $A$ by $V$ satisfies
\begin{equation}\label{eq:diffhitting}\frac{d}{dt}\Prp{v}{\tau^V_A\leq t} = \left\{\begin{array}{ll}0, & v\in A;\\ \sum_{w\in V}q(v,w)\,(\Prp{w}{\tau^V_A\leq t} -\Prp{v}{\tau^V_A\leq t}), & v\in \Omega_V\backslash A.\end{array}\right.\end{equation}
The derivative is understood as a right derivative at time $t=0$.

Let $\Delta = \{(x,x)\,:\,x\in\Omega\} \subset \Omega^2$ be the diagonal set. 
We first apply \eqnref{diffhitting} to the product chain $(V_t)_{t\geq 0} = (Y_t,Z_t)_{t\geq 0}$, with $\Omega_V=\Omega^2$, and $A=\Delta$. In this case $\tau^V_A=M^{Y,Z}$, and a straightforward computation with the transition rates gives:
\begin{equation}\label{eq:systemforf}\frac{d}{dt}\,f_{x,z}(t)  = \left\{\begin{array}{ll}0,& x=z\\ \lambda_Y\sum_{x'\in \Omega}P(x,x')\,(f_{(x',z)}(t)-f_{(x,z)}(t)) & \\ + \lambda_Z\sum_{z'\in \Omega}P(z,z')\,(f_{(x,z')}(t)-f_{(x,z)}(t)), & x\neq z.\end{array}\right.\end{equation}
We now apply the same formula \eqnref{diffhitting} with $(V_t)_{t\geq 0}=(X_t)_{t\geq 0}$. Note that $g_{(x,z)}(t):=\Prp{x}{\tau_z\leq s(t)}$ where $s(t) = (\lambda_Y+\lambda_Z)\,t$, so the chain rule implies
\begin{equation}\label{eq:derivativetau1}\frac{d}{dt}\,g_{(x,z)}(t)= \left\{\begin{array}{ll}0,& x=z\\ (\lambda_Y+\lambda_Z)\,\sum_{x'\in \Omega}P(x,x')\,(g_{(x',z)}(t) - g_{(x,z)}(t)), & x\neq z.\end{array}\right.\end{equation}
{\em We will now make crucial use of transitivity}, which allows us to use \lemref{knownfacts} to deduce that $\Prp{x}{\tau^X_z\leq (\lambda_Y+\lambda_Z)\,t}$ is symmetric in $x$ and $z$, i.e.
$$\Prp{x}{\tau^X_z\leq (\lambda_Y+\lambda_Z)\,t} = \Prp{z}{\tau^X_x\leq (\lambda_Y+\lambda_Z)\,t},$$ that is $g_{x,z}(\cdot) = g_{z,x}(\cdot)$ for all $x,z$. This allows us to reverse the roles of $x $ and $z$ in \eqnref{derivativetau1} to obtain:
\begin{equation}\label{eq:derivativetau2}\frac{d}{dt}\,g_{(x,z)}(t)= \left\{\begin{array}{ll}0,& x=z\\ (\lambda_Y+\lambda_Z)\,\sum_{z'\in \Omega}P(z,z')\,(g_{(x,z')}(t) - g_{(x,z)}(t)), & x\neq z.\end{array}\right.\end{equation}
We add the two formulas \eqnref{derivativetau1} and \eqnref{derivativetau2} with weights $\lambda_Y/(\lambda_Y+\lambda_Z)$ and $\lambda_Z/(\lambda_Y+\lambda_Z)$ respectively. The upshot is:
\begin{equation*}
\frac{d}{dt}\,g_{(x,z)}(t)  = \left\{\begin{array}{ll}0,& x=z\\ \lambda_Y\sum_{x'\in \Omega}P(x,x')\,(g_{(x',z)}(t)-g_{(x,z)}(t)) & \\ + \lambda_Z\sum_{z'\in \Omega}P(z,z')\,(g_{(x,z')}(t)-g_{(x,z)}(t)), & x\neq z.\end{array}\right.
\end{equation*}
This is precisely the system of ODEs we obtained for the $f$'s in \eqnref{systemforf} and it concludes the proof.
\end{proof}

\section{The general lower bound}\label{sec:redmond}

In this section we prove \thmref{redmond}.

\begin{proof}[\bf Proof of \thmref{redmond}] 

We let $n=|\Omega|$ denote the number of states. The transitivity assumption implies $\pi(v)=1/n$ for all $v\in\Omega$. The next lemma will be used in the proof. We defer its proof until Section~\ref{sec:occupation}.

\begin{lemma}\label{lem:upperboundRHS} 
Let $P$ be a reversible and transitive transition matrix. 
Let $X, Y$ and $Z$ be three independent continuous time chains with transition matrix $P$ and speeds $\lambda_X=1$ and $\lambda_Y,\lambda_Z\geq 0$. Let $\mu$ be the probability measure given by
\begin{align}\label{eq:defmu}
\mu(A):= \prcond{(X_{\Mgood}, Y_{\Mgood}, Z_{\Mgood})\in A}{\Mgood<\Mbad}{} \quad \text{for } A\subseteq \Omega\times \Omega\times \Omega.
\end{align}
Then
$$\int_{0}^{+\infty}\Prp{\mu}{X_t=Y_t,t<\Mbad}\,dt\leq \frac{22\,t_{\rm hit}}{n}.$$\end{lemma}

Our proof is based on the analysis of the time that $X,Y$ spend on the diagonal $\Delta=\{(x,x): \, x\in \Omega\}$ prior to time $\Mbad$, i.e.
\begin{equation}\label{eq:defT}T:=\int_{0}^{\Mbad}\1(X_t=Y_t)\,dt = \int_{0}^{\infty}\1(X_t=Y_t,t<\Mbad)\,dt.\end{equation}
Note that $\Mgood<\Mbad$ if and only if $T>0$, so
\begin{align}\label{eq:longet}
\nonumber\E{T} &= \escond{T}{\Mgood<\Mbad}{}\pr{\Mgood<\Mbad}\\\nonumber &=\estart{\int_0^\infty \1(X_t=Y_t, t<\Mbad)\,dt}{\mu} \pr{\Mgood<\Mbad}\\
&=\pr{\Mgood<\Mbad}\int_0^\infty \prstart{X_t=Y_t,t<\Mbad}{\mu}\,dt. 
\end{align}
Lemma~\ref{lem:upperboundRHS} upper bounds the integral appearing above. Thus in order to obtain a lower bound for $\prstart{\Mgood<\Mbad}{}$ it suffices to lower bound 
\begin{align}\label{eq:ExpT}
\estart{T}{}= \int_0^\infty \prstart{X_t=Y_t, t<\Mbad}{}\,dt.
\end{align}
Using reversibility and the fact that $\pi$ is uniform for a transitive chain we obtain
\begin{align*}
\prstart{X_t=Y_t,t<\Mbad}{} &= \Prp{}{X_t=Y_t,\forall s\leq t,\, X_s\neq Z_s\mbox{ and }Y_s\neq Z_s} \\&=\Prp{}{X_0=Y_0,\forall s\leq t,\, X_s\neq Z_s\mbox{ and }Y_s\neq Z_s}\\
&=\sum_{x\in\Omega,z\in\Omega\backslash\{x\}}\frac{\Prp{(x,x,z)}{\forall s\leq t,\, X_s\neq Z_s\mbox{ and }Y_s\neq Z_s}}{n^3}\\
&=\sum_{x\in\Omega,z\in\Omega\backslash\{x\}}\frac{\Prp{(x,x,z)}{\Mbad>t}}{n^3} =  \frac{1}{n}\sum_{x,z\in\Omega}\pi(x)\,\pi(z)\,\Prp{(x,x,z)}{\Mbad>t}.
\end{align*}

Plugging this back into \eqnref{ExpT} gives
\begin{equation}\label{eq:ExpT2}\Exp{}{T} = \frac{1}{n}\sum_{x\in\Omega}\,\pi(x)\left(\int_{0}^{\infty}\sum_{z\in\Omega}\pi(z)\,\Prp{(x,x,z)}{\Mbad>t}\,dt\right).\end{equation}

At this point we recall that $\Mbad=M^{X,Y}\wedge M^{Y,Z}$, therefore
$$\forall (x,z)\in \Omega^2\,;\, \Prp{(x,x,z)}{\Mbad>t}\geq 1 -\Prp{(x,z)}{M^{X,Z}\leq t} - \Prp{(x,z)}{M^{Y,Z}\leq t}.$$
By \corref{smalltimetransitive}, for each $x\in\Omega$ there exists a subset $A_x$ with $\pi(A_x)\geq 1/2$ for which we have the following bound for all $t>0$
$$\sum_{z\in A_x}\pi(z)\Prp{(x,x,z)}{\Mbad>t}\geq \pi(A_x)\,\left[1 - 6(\sqrt{1+\lambda_Z} + \sqrt{\lambda_Y+\lambda_Z})\,\sqrt{\frac{t}{t_{\rm hit}}}\right].$$
Therefore, fixing $\xi>0$ and $x\in\Omega$:
\begin{eqnarray*}\int_{0}^{\infty}\sum_{z\in\Omega}\pi(z)\,\Prp{(x,x,z)}{\Mbad>t}\,dt &\geq &\frac{1}{2}\,\int_{0}^\xi \sum_{z\in A_x}\frac{\pi(z)}{\pi(A_x)}\,\Prp{(x,x,z)}{\Mbad>t}\,dt\\ &\geq & \frac{1}{2}\int_0^\xi\left[1 - 6(\sqrt{1+\lambda_Z} + \sqrt{\lambda_Y+\lambda_Z})\,\sqrt{\frac{t}{t_{\rm hit}}}\right]\,dt\\ &=& \frac{\xi}{2}\left[1 - 4\,(\sqrt{1+\lambda_Z} + \sqrt{\lambda_Y+\lambda_Z})\,\sqrt{\frac{\xi}{t_{\rm hit}}}\right].\end{eqnarray*}
We can maximize the right hand side by taking 
$$\xi := \frac{t_{\rm hit}}{36\,(\sqrt{1+\lambda_Z} + \sqrt{\lambda_Y+\lambda_Z})^2},$$
which gives the bound
$$\forall x\in\Omega\,:\, \int_{0}^{\infty}\sum_{z\in\Omega}\pi(z)\,\Prp{(x,x,z)}{\Mbad>t}\,dt \geq \frac{t_{\rm hit}}{216\,(\sqrt{1+\lambda_Z} + \sqrt{\lambda_Y+\lambda_Z})^2}.$$
Combining this with \eqnref{ExpT2} gives that 
\begin{align*}
\estart{T}{} \geq  \frac{t_{\rm hit}}{216\,n\,(\sqrt{1+\lambda_Z} + \sqrt{\lambda_Y+\lambda_Z})^2}.
\end{align*}
Therefore, using this together with Lemma~\ref{lem:upperboundRHS} and~\eqref{eq:longet} we conclude
\begin{align*}
\prstart{\Mgood<\Mbad}{}\geq \frac{1}{4752(\sqrt{1+\lambda_Z} + \sqrt{\lambda_Y+\lambda_Z})^2}
\end{align*}
and this finishes the proof.
\end{proof}

\subsection{Occupation time up to a stopping time}\label{sec:occupation}

The goal of this section is to prove Lemma~\ref{lem:upperboundRHS}. We start with a  more general setting. We give the proof of the lemma at the end of the section.

It is well known that a finite irreducible chain $(V_t)_{t\geq 0}$ with state space $\Omega_V$, started from a point $x$ and stopped at a stopping time $\tau>0$ with $V_\tau=x$, satisfies:
$$\forall v\in\Omega_V\,:\, \Exp{x}{\int_0^{\tau}\1(X_t=v)\,dt} = \pi_V(v)\,\Exp{x}{\tau},$$
where $\pi_V$ is the unique stationary measure of $(V_t)_{t\geq 0}$ (some simple conditions on $\tau$ are necessary for this). There are also extensions of this lemma to the case where $V_0$ and $V_\tau$ are not necessarily equal, but have the same distribution \cite[Proposition 2.4, Chapter 2]{AldFill}. The following lemma extends this idea even further, and shows that $\tau$ may be a stopping time for a ``larger"~Markov chain.

\begin{lemma}\label{lem:stop}Suppose $(U_t)_{t\geq 0}$, $(V_t)_{t\geq 0}$ are independent, irreducible, continuous time Markov chains with finite state spaces $\Omega_U$ and $\Omega_V$ respectively. 
Assume $\mu$ is a probability measure over $\Omega_U\times \Omega_V$ and that $\tau$ is a stopping time for the process $(U_t,V_t)_{t\geq 0} $ with the following properties.
\begin{enumerate}
\item[\rm (1)] $\Prpwo{\mu}(\tau>0)=1$;
\item [\rm (2)] $\Exp{\mu}{\tau}<\infty$.
\item [\rm (3)] $\Prp{\mu}{V_0=\cdot} = \Prp{\mu}{V_\tau=\cdot}.$
\end{enumerate}  Then for all $v\in \Omega_V$
$$\Exp{\mu}{\int_{0}^{\tau}\1(V_t=v)\,dt} = \Exp{\mu}{\tau}\,\pi_V(v)$$
where $\pi_V$ is the stationary distribution of $V$.
\end{lemma}

\begin{proof}[\bf Proof]
In this proof we will interchange integrals, expectations and summations several times. Instead of justifying this at each step, we note right away that all of these interchanges are valid, because the integrands are non-negative. 

Consider the row vector $h$ with nonnegative coordinates
$$h(v):=\Exp{\mu}{\int_{0}^{\tau}\1(V_t=v)\,dt}=\int_0^{\infty}\,\Prp{\mu}{V_t=v,\tau>t}\,dt\;\;\;(v\in \Omega_V).$$
Note that $\sum_{v\in V}h(v) = \Exp{\mu}{\tau}>0$ because $\tau>0$ a.s.. Letting $Q$ be the generator of $(V_t)_{t\geq 0}$, we will show below that 
\begin{equation}\label{eq:goallocal}
hQ=0.
\end{equation} 
This identity implies that $h/\Exp{\mu}{\tau}$ is one invariant probability distribution for $V$. Since $\pi_V$ is the unique invariant distribution, we deduce that for all $v\in \Omega_V$
$$\frac{h(v)}{\Exp{\mu}{\tau}} = \pi_V(v),$$
which is precisely what we need to prove. 

We will derive $hQ=0$ from the limit
\begin{equation}\label{eq:limitformhQ}\forall v\in\Omega_V\,:\, h\,Q(v) = \lim_{\eps\searrow 0}\frac{[h\,e^{\eps\,Q}](v) - h(v)}{\eps}.\end{equation}
In order to compute the limit we recall $e^{\eps\,Q}(w,v) = \Prp{w}{X_\eps=v}$ for all $w,v\in\Omega_V$. Therefore
$$[h\,e^{\eps\,Q}](v) = \sum_{w\in \Omega_V}\,h(w)\,\Prp{w}{V_\eps=v} = \int_0^{\infty}\,\sum_{w\in \Omega_V}\Prp{\mu}{V_t=w,\tau>t}\,\Prp{w}{V_\eps=v}\,dt.$$
Crucially, the fact that $\tau$ is a stopping time for $U,V$ implies that the event $\{V_t=w,\tau>t\}$ is measurable with respect to $(U_s,V_s)_{s\leq t}$. Using that $V$ and $U$ evolve independently and the Markov property for $V$ implies 
\begin{eqnarray*}\Prp{\mu}{V_t=w,\tau>t}\,\Prp{w}{V_\eps=v} &=& \Prp{\mu}{V_t=w,\tau>t}\,\Prp{\mu}{V_{t+\eps}=v\mid V_t=w,\tau>t} \\ &=& \Prp{\mu}{V_t=w,V_{t+\eps}=v,\tau>t}.\end{eqnarray*}
Plugging this back in the previous display gives
\begin{eqnarray}\nonumber [h\,e^{\eps\,Q}](v) &=& \int_0^{\infty}\,\Prp{\mu}{V_{t+\eps}=v,\tau>t}dt \\ \nonumber &=& \int_0^{\infty}\,\Prp{\mu}{V_{t+\eps}=v,\tau>t+\eps}\,dt  + \int_0^{\infty}\,\Prp{\mu}{V_{t+\eps}=v,t\leq \tau\leq t+\eps}\,dt\\ \label{eq:twoterms} &=:& (I)+(II).\end{eqnarray}
The first term is
$$(I) =  \int_\eps^{\infty}\,\Prp{\mu}{V_{t}=v,\tau>t}\,dt  = h(v) - \int_{0}^{\eps}\Prp{\mu}{V_{t}=v,\tau>t}\,dt,$$
so
\begin{equation}\label{eq:limitI}\frac{(I) - h(v)}{\eps}\to -\Prp{\mu}{V_0=\mu,\tau>0} = -\Prp{\mu}{V_0=v}\end{equation}
because $\tau>0$ always. 
Regarding the second term, we have
\begin{align}\label{eq:iiterm}
(II) =  \int_0^\infty\prstart{V_t=v, V_\tau=v, \tau\leq t\leq \tau+\epsilon}{\mu}\,dt +\int_0^\infty\prstart{V_t=v, V_\tau\neq v, \tau\leq t\leq \tau+\epsilon}{\mu}\,dt. 
\end{align}
For the first term on the right hand side above we obtain
\begin{align}\label{eq:firsterm}
\lim_{\epsilon\to 0} \frac{\int_0^\infty\prstart{V_t=v, V_\tau=v, \tau\leq t\leq  \tau+\epsilon}{\mu}\,dt}{\epsilon} = \prstart{V_\tau=v}{\mu}.
\end{align} 
As for the second term in the sum in~\eqref{eq:iiterm} we get
\begin{align*}
\nonumber&\int_0^\infty\prstart{V_t=v, V_\tau\neq v, \tau\leq t\leq \tau+\epsilon}{\mu}\,dt
= \int_0^\infty \estart{\prcond{V_t=v,V_\tau\neq v, \tau\leq t\leq \tau+\epsilon}{\tau}{\mu}}{\mu} \,dt \\
&= \int_0^\infty \estart{\1(\tau\leq t\leq \tau+\epsilon) \prcond{V_t=v,V_\tau\neq v}{\tau}{\mu}}{\mu} \,dt.
\end{align*}
On the event $\{\tau\leq t\leq \tau+\epsilon\}$ in order 
to have $V_t=v$ and $V_\tau\neq v$, there must exist at least one jump of the Markov chain in the time interval $[\tau,t]$, which on this event has length less than $\epsilon$.
Therefore, we  obtain that  on the event $\{\tau\leq t\leq \tau+\epsilon\}$
\begin{align*}
\prcond{V_t=v,V_\tau\neq v}{\tau}{\mu} = O(\epsilon).
\end{align*}
Therefore we deduce
\begin{align*}
\int_0^\infty \estart{\1(\tau\leq t\leq \tau+\epsilon) \prcond{V_t=v,V_\tau\neq v}{\tau}{\mu}}{\mu}\,dt = O(\epsilon^2).
\end{align*}
Hence this together with~\eqref{eq:firsterm} gives that 
$$\frac{(II)}{\eps}\to  \Prp{\mu}{V_{\tau}=v}\mbox{ as }\eps\searrow 0.$$
Combining this with \eqnref{limitI} and \eqnref{twoterms} gives:
$$\frac{[h\,e^{\eps\,Q}](v) - h(v)}{\eps}\to \Prp{\mu}{V_{\tau}=v} - \Prp{\mu}{V_{0}=v}$$
Our assumption that $\Prp{\mu}{V_0=\cdot} = \Prp{\mu}{V_\tau=\cdot}$ implies that the right hand side above is zero. Plugging this back into \eqnref{limitformhQ} gives $hQ=0$ and finishes the proof.\end{proof}

\begin{proof}[\bf Proof of Lemma~\ref{lem:upperboundRHS}]

We want to eventually use \lemref{stop} to estimate the integral, which we can rewrite as 
\begin{equation}\label{eq:rewriteintegral}\int_{0}^{\infty}\Prp{\mu}{X_t=Y_t,t<\Mbad}\,dt=\sum_{x\in\Omega}\Exp{\mu}{\int_{0}^{\Mbad}\1(X_t=Y_t=x)\,dt}.\end{equation}
This is a sum of terms of the form demanded by \lemref{stop}: just set $(U_t)_{t\geq 0}=(Z_t)_{t\geq 0}$ and $(V_t)_{t\geq 0} = (X_t,Y_t)_{t\geq 0}$. However, other conditions are needed for this lemma to apply, and one of them clearly fails: the distribution of $(X_0,Y_0)$ is  {\em  not the same} as that of $(X_{\Mbad},Y_{\Mbad})$. To see this, simply note that whereas $X_0=Y_0$ under $\mu$ (as we will see below), typically $X_{\Mbad}\neq Y_{\Mbad}$. It turns out that we can circumvent this problem by defining 
\begin{equation}\label{eq:deftau}
\tau = \inf\{t\geq \Mbad: \, X_t=Y_t\}.
\end{equation}
Note that $X_t\neq Y_t$ for all $\Mbad\leq t<\tau$. Therefore
\begin{equation}\int_{0}^{\infty}\Prp{\mu}{X_t=Y_t,t<\Mbad}\,dt=\sum_{x\in\Omega}\Exp{\mu}{\int_{0}^{\tau}\1(X_t=Y_t=x)\,dt}.\end{equation}
Clearly $\tau$ is a stopping time for $(X_t,Y_t,Z_t)_{t\geq 0}$. We claim that $\tau$ and the initial distribution $\mu$ satisfy the conditions $(1)-(3)$ of \lemref{stop}. 

First, $\prstart{\tau>0}{\mu}=1$, since $\prstart{X_0=Y_0\neq Z_0}{\mu}=1$ and $\prstart{\tau\geq \Mbad>0}{\mu}=1$. Hence condition~(1) of Lemma~\ref{lem:stop} is satisfied. Next we show that 
\begin{align}\label{eq:thitstartbound}
\estart{\tau}{\mu}\leq 22t_{\rm hit}^*,
\end{align}
which will imply that condition (2) is also satisfied. 
Recalling that $\tau = \inf\{t\geq \Mbad: \, X_t=Y_t\}$, we have
$$\Exp{\mu}{\tau}\leq \max_{(x,y,z)}\Exp{(x,y,z)}{\Mbad} + \max_{(x',y',z')}\Exp{(x',y',z')}{\Mgood}.$$
The first expectation is at most the meeting time of~$X$ and~$Z$, which is independent of $Y$. \lemref{meetingtime} implies that, conditionally on $Z$, this is at most $11t^{*}_{\rm hit}$ almost surely, so $\Exp{(x,y,z)}{\Mbad} \leq 11t^{*}_{\rm hit}$. Similarly, $\Mgood$ is the meeting time of $X$ and the independent trajectory $Y$, and $\Exp{(x',y',z')}{\Mgood}\leq 11t^{*}_{\rm hit}$.

Note now that by definition $\mu$ is supported on the set $\{(x,x,z): \, x\neq z\}$. One can also check that $\mu$ is invariant under the action of any automorphism $\phi$ of $P$, i.e.\
$\mu(x,y,z) = \mu(\phi(x),\phi(y),\phi(z))$. This together with transitivity give that the marginal of $\mu$ on the first two coordinates is uniform on the diagonal set~$\Delta = \{(x,x)\in \Omega^2: x\in \Omega\}$. 

Similarly the law of $(X_\tau,Y_\tau)$ is invariant under the action of any automorphism $\phi$.
Using transitivity again we obtain that $(X_{\tau}, Y_\tau)$ is uniform on $\Delta$. Therefore, condition~(3) is also satisfied.

We can now apply Lemma~\ref{lem:stop} to get that 
\begin{align*}
\int_{0}^{\infty}\Prp{\mu}{X_t=Y_t,t<\Mbad}\,dt= \sum_{x\in\Omega}\frac{\Exp{\mu}{\tau}}{n^2}\leq \frac{22\,t^{*}_{\rm hit}}{n},
\end{align*}
where the inequality follows from~\eqref{eq:thitstartbound}. This concludes the proof.
\end{proof}

\section{Non transitive chains}

The goal of this section is to prove Theorem~\ref{thm:counter}. 
Throughout the section we fix $\epsilon>0$ and let~$C=6/\epsilon$.  In what follows $K_r$ is the complete graph on~$r\in \N\setminus\{0\}$ vertices. 

For $n\in \N$ construct a graph $G_n$ as follows: 
begin from a clique $K_{n+1}$ and $n$ disjoint copies of $K_k$ with $k=\sqrt{Cn}$. Fix a vertex $v\in K_{n+1}$ and add exactly one edge from $v$ to each copy of $K_k$. See Figure~\ref{fig:graphG} for a depiction of the graph.

\begin{figure}[h!]
\begin{center}
\includegraphics[scale=1]{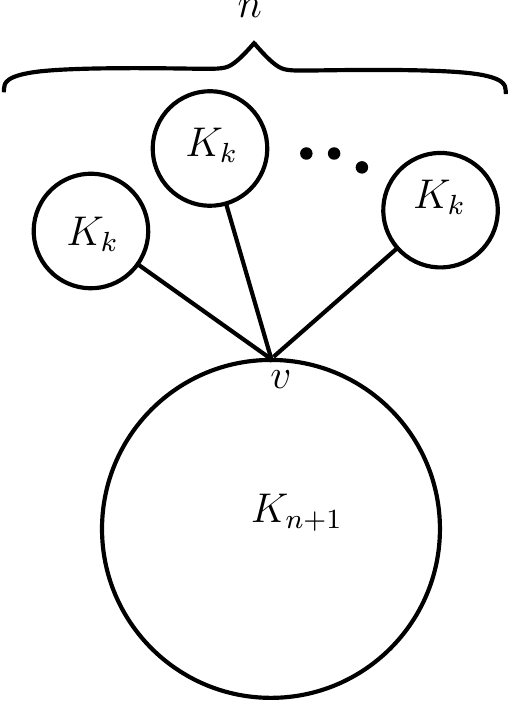}
\caption{\label{fig:graphG}The graph $G$}
\end{center}
\end{figure}

Let $\Omega$ be the vertex set of $G_n$. We call $\text{Down}$ the set of vertices belonging to $K_{n+1}$ and $\text{Up}= \Omega\setminus\text{Down}$ the rest.  

Let $P$ be the transition matrix of a simple random walk over $G$ and $\pi$ its stationary distribution. Let~$X, Y$ and $Z$ be independent random walks starting from $\pi$ with transition matrix $P$ and speeds $\lambda_X=1, \lambda_Y=0$ and $\lambda_Z=1$.

The idea is that by choosing $\epsilon$ sufficiently small, the stationary measure of $\text{Down}$ becomes arbitrarily small.  So if we start $X, Y$ and $Z$ according to $\pi$, then it is very likely they will all start from different cliques in Up. Let $T$ be the $\sqrt{n}$-th time that $X$ visits the vertex $v$. We will show that as~$n\to \infty$ the probability that $X$ and $Z$ collide after time $T$ is arbitrarily small. Moreover, we will show that the probability that $X$ and $Y$ collide before $T$ is arbitrarily small as $n\to \infty$. Combining these two assertions will complete the proof.

For all $r\geq 0$ we define $\tau_v^{(r)}$ to be the time of the $r$-th visit to $v$. Formally,
\[
\tau_v^{(0)} = \inf\{t\geq 0: \,X_t=v\}
\]
and for $i\geq 1$ we define
\[
\tau_v^{(i)} = \inf\{t>\tau_v^{(i-1)}: \, X_t=v, \, X_{t-}\neq v\}.
\]

\begin{lemma}\label{lem:firstalpha}
There exists $\alpha=\alpha(C)>0$ independent of $n$ such that 
for all $x,z\in \Omega$ and all $r\geq 1$ we have 
\[
\prstart{M^{X,Z}> \tau_v^{(r)}}{x,z} \leq (1-\alpha)^{r-1}.
\]
\end{lemma}

\begin{proof}[\bf Proof]

First note that by the strong Markov property we have for all $r\geq 1$
\begin{align*}
\sup_{x,z}\prstart{M^{X,Z}>\tau_v^{(r)}}{x,z} \leq \sup_{z}\prstart{M^{X,Z}>\tau_v^{(r-1)}}{v,z}.
\end{align*}
Using the strong Markov property again, for $r\geq 1$ we obtain
\begin{align*}
\sup_z \prstart{M^{X,Z}>\tau_v^{(r)}}{v,z} &= \sup_z \prcond{M^{X,Z}>\tau_v^{(r)}}{M^{X,Z}>\tau_v^{(1)}}{v,z}\prstart{M^{X,Z}>\tau_v^{(1)}}{v,z} \\&\leq \sup_w \prstart{M^{X,Z}>\tau_v^{(r-1)}}{v,w} \sup_z\prstart{M^{X,Z}>\tau_v^{(1)}}{v,z}.
\end{align*}
By induction for all $r\geq 1$ this yields
\begin{align*}
\sup_z \prstart{M^{X,Z}>\tau_v^{(r)}}{v,z} \leq \left( \sup_z\prstart{M^{X,Z}>\tau_v^{(1)}}{v,z}\right)^r.
\end{align*}
So we complete the proof by showing that 
\begin{align}\label{eq:newgoal}
\sup_z\prstart{M^{X,Z}>\tau_v^{(1)}}{v,z} \leq 1-\alpha
\end{align}
for a positive constant $\alpha$ depending only on $C$.

Let $\tau=\inf\{t\geq 0: Z_t\in {\rm{Down}}\setminus\{v\}\}$ and fix $w\in {\rm{Down}}\setminus\{v\}$. By symmetry, for all $z$ we then have
\begin{align*}
\prstart{M^{X,Z}\leq \tau_v^{(1)}}{v,z} \geq \frac{1}{2}\cdot \min_z\prstart{\tau\leq \tau_v^{(0)}}{w,z} \min_{a,b\,\in\, {\rm{Down}}\setminus\{v\}}\prstart{M^{X,Z}\leq \tau_v^{(0)}}{a,b},
\end{align*}
where the factor $1/2$ corresponds to the probability that the first time $X$ jumps it goes to ${\rm{Down}}\setminus\{v\}$. 

If $X_0=a \in {\rm{Down}}\setminus \{v\}$, then $\tau_v^{(0)} = \tau_v^{X}$, and hence if also $b\in {\rm{Down}}\setminus \{v\}$, then 
\begin{align*}
\prstart{M^{X,Z}\leq \tau_v^{(0)}}{a,b} \geq \prstart{M^{X,Z}\leq \tau_v^X\wedge \tau_v^Z}{a,b} =\frac{1}{2}.
\end{align*}

It remains to show that for a positive constant $c_1$ we have
\begin{align}\label{eq:goalsecondterm}
\min_z\prstart{\tau\leq \tau_v^{(0)}}{w,z} \geq c_1>0.
\end{align}

If $z\in {\rm{Down}}\setminus\{v\}$, then this probability is $1$ and if $z=v$ it is easily seen to be at least $1/4$. So we assume that $z\in \rm{Up}$. Let~$x$ be the unique neighbour of $v$ lying in the same clique as $z$.
Then the time $\tau$ can be expressed as~$\tau = T_{z,x} + T_{x,v}+T_{v,{\rm{Down}}\setminus\{v\}}$, where the time $T_{r,S}$ stands for the first hitting time of $S$ starting from~$r$. Using this, it is then not hard to see that there exists a positive constant $c$ such that uniformly over all $z\in \rm{Up}$ we have $\E{\tau}\leq c k^2$.
Moreover, if $X_0\in {\rm{Down}}\setminus\{v\}$, then~$\tau_v^{(0)}$ is an exponential random variable with mean $n$.
By Markov's inequality we obtain
\begin{align*}
\prstart{\tau\leq \tau_v^{(0)}}{w,z}  \geq \prstart{\tau\leq 2\E{\tau}}{z} \cdot \prstart{\tau_v^{(0)}\geq 2\E{\tau}}{w} \geq \frac{1}{2}\cdot \int_{2\E{\tau}}^{\infty} ne^{-n s}\,ds = \frac{1}{2} e^{-2n\E{\tau}}.
\end{align*}
Note that this bound does not depend on $z$. Since $k=\sqrt{Cn}$ and $\E{\tau}\leq c k^2$ the bound in~\eqref{eq:goalsecondterm} follows.
\end{proof}

\begin{proof}[\bf Proof of Theorem~\ref{thm:counter}]

We show that for $n$ sufficiently large, the graph $G=G_n$ satisfies the claim of the theorem.

It is not hard to verify that for $n$ large enough, in $G_n$ we have 
\[
\pi({\rm{Down}})\leq \frac{2}{C}.
\]
Let $A$ be the set of pairs $(x,y)$ such that $y\in \text{Up}$ and $x$ is not in the same clique as $y$. Then let~$E=\{(X_0,Y_0)\in A\}$. 
By the preceding bound $\pr{E^c}\leq 3/C= \epsilon/2$ for $n$ large enough.
We then have
\begin{align}\label{eq:uppergood}
\nonumber\pr{\Mgood\leq \Mbad} &\leq \pr{M^{X,Y}\leq M^{X,Z}} = \pr{M^{X,Y}\leq M^{X,Z}, E} + \pr{M^{X,Y}\leq M^{X,Z}, E^c} \\&\leq \sup_{x,y,z: \,(x,y)\in A}\prstart{M^{X,Y}\leq M^{X,Z}}{x,y,z} + \frac{\epsilon}{2}.
\end{align}
Therefore, it suffices to upper bound the last probability appearing above. Fix any $z\in \Omega$ and $(x,y)\in A$. 
For $r$ to be determined later we have
\begin{align}\label{eq:othergood}
\prstart{M^{X,Y}\leq  M^{X,Z}}{x,y,z} \leq \prstart{M^{X,Y}\leq \tau_v^{(r)}}{x,y} + \prstart{M^{X,Z}>\tau_v^{(r)}}{x,z}.
\end{align}
Because $Y$ is not moving, we have $M^{X,Y}=\tau_y$, where $y=Y_0$. Since $x,y$ are not in the same clique of~Up,  if $\tau_y\leq \tau_v^{(r)}$, then there exists $1\leq i\leq r$ such that $\tau_v^{(i-1)}<\tau_y\leq \tau_v^{(i)}$.
By the strong Markov property and union bound we obtain
\begin{align*}
\prstart{M^{X,Y}\leq \tau_v^{(r)}}{x,y} \leq r\prstart{\tau_y\leq \tau_v^{(1)}}{v} \leq \frac{r}{2n},
\end{align*}
since, when $X_0=v$, in order to hit $y\in\rm{Up}$ before returning to $v$,  the first time $X$ moves it must jump into the clique that contains $y$.

Using the above bound and Lemma~\ref{lem:firstalpha} in~\eqref{eq:othergood} we deduce
\begin{align*}
\prstart{M^{X,Y}\leq  M^{X,Z}}{x,y,z} \leq \frac{r}{2n} + (1-\alpha)^{r-1}.
\end{align*}
Taking $r=\sqrt{n}$ or any other function of $n$ that goes to infinity slower than $n$ gives that 
\[
\prstart{M^{X,Y}\leq  M^{X,Z}}{x,y,z} \to 0 \quad \text{as } n\to \infty.
\]
We conclude from~\eqref{eq:uppergood} that
\[
\pr{\Mgood\leq \Mbad} <\epsilon
\]
and this finishes the proof.
\end{proof}

\section{Sharpness of Conjecture~\ref{conj:1} }\label{sec:sharpness}

In this section we describe the example pointed out by Alexander Holroyd, mentioned in the Introduction, of a family of transitive graphs for which $\pr{\Mgood \leq \Mbad}\leq 1/3+\delta$. 
In what follows we take $\lambda_X=\lambda_Y=1$ and $\lambda_Z=0$. 

To construct the example, fix $\eps \in (0,1)$ and consider the chain with state space $\{0,1\}^n$ in which the $j$'th coordinate changes value (from $0$ to $1$ or vice-versa) at rate $q_j = \eps^{j-1}(1-\eps)/(1-\eps^n)$; note that $\sum_{i=1}^n q_i = 1$. The idea is that for small $\eps$, earlier coordinates change state  much more quickly than later coordinates, so the primary obstacle to both meeting and hitting is simply the largest coordinate in which the value differs. For $u,v \in \{0,1\}^n$, let $k(u,v) = \max\{i: u_i \ne v_i\}$, or $k(u,v)=0$ if $u=v$. 

We claim that for $x,y,z \in \{0,1\}^n$, if $k(x,y) > \min(k(x,z),k(y,z))$ then $\prstart{\Mgood < \Mbad}{x,y,z} < 2\eps n$. Assuming this, and taking $\eps=\delta/(2n)$, it follows by symmetry that, starting from stationarity, 
\[
\pr{\Mgood< \Mbad} \le \delta + \pr{k(X_0,Y_0) < \min(k(X_0,Z_0),k(Y_0,Z_0)} < \delta + \frac{1}{3}.
\]
It thus remains to prove the preceding claim. 

Fix $x,y,z \in \{0,1\}^n$ with $k(x,y) > \min(k(x,z),k(y,z))$, and assume by symmetry that $k(x,z) < k(x,y)$. 
For $1 \le k \le n$, let $\tau_k = \min\{t: X_t^{(i)}=z_i,1 \le i \le k\}$ be the first time that $X_t$ and $z$ agree in the first $k$ coordinates. It is convenient to set $\tau_0=0$. 
Also, let $\sigma_k^X = \min\{t: \exists i \ge k,~X_t^{(i)} \ne X_0^{(i)}\}$ be the first time one of the last $n-k+1$ coordinates of $X$ changes, and define $\sigma_k^Y$ accordingly. 

We will show that for all $1 \le k < n$, 
\begin{equation}\label{eq:a_claim}
\pr{\tau_k < \sigma_{k+1}^X} \ge (1-\eps)^k \geq  1-k\eps. 
\end{equation}
Note that $\tau_k$, $\sigma_{k+1}^X$ and $\sigma_{k+1}^Y$ are all independent. Furthermore, $\sigma_{k+1}^X$ and $\sigma_{k+1}^Y$ are identically distributed, so if the preceding inequality holds as written then it also holds with $\sigma_{k+1}^Y$ in place of~$\sigma_{k+1}^X$. We finish proving the claim assuming that (\ref{eq:a_claim}) holds, then conclude by proving (\ref{eq:a_claim}).

At time $\tau_{k(x,z)}$, the first $k(x,z)$ coordinates of $X$ agree with those of $z$. If $\tau_{k(x,z)} < \sigma_{k(x,z)+1}^X$ then the remaining coordinates of $X$ and $z$ also agree (because they did at time $0$ and they have not changed), so $M^{X,Z}=\tau_{k(x,z)}$. Similarly, if $\tau_{k(x,z)} < \sigma_{k(x,z)+1}^X$ and $\tau_{k(x,z)} < \sigma_{k(x,z)+1}^Y$ then $\Mbad < \Mgood$.
It then follows, using (\ref{eq:a_claim}) and the subsequent observation, that  
\[
\prstart{\Mgood < \Mbad}{x,y,z} \le \pr{\sigma_{k(x,z)+1}^X\le \tau_{k(x,z)}} + \pr{\sigma_{k(x,z)+1}^Y\le \tau_{k(x,z)}} \le 2k(x,z)\eps < 2n\eps\,,
\]
as claimed. It thus remains to prove (\ref{eq:a_claim}). In what follows we write $\sigma_k=\sigma_k^X$. 

Fix $1 \le k < n$, and note that $\sigma_k$ is exponential with rate $\sum_{j=k}^n q_j = \eps^{k-1}(1-\eps^{n+1-k})/(1-\eps^n)$. 
Furthermore, $\sigma_k < \sigma_{k+1}$ precisely if the $k$'th coordinate of $X$ changes before any larger coordinate. It follows that $\pr{\sigma_k < \sigma_{k+1}} = q_k/\sum_{j=k}^n q_j$. 

Suppose that $X_0^{(k)}=z_k$. In this case to have $\tau_k < \sigma_{k+1}$ it suffices that $\tau_{k-1} < \sigma_k$, so 
\[
\prcond{\tau_k < \sigma_{k+1}}{X_0^{(k)}=z_k}{} \geq  \pr{\tau_{k-1} < \sigma_k}. 
\]
If $X_0^{(k)}\ne z_k$ then the $k$'th coordinate must change before time $\tau_k$, so to have $\tau_k < \sigma_{k+1}$ it is necessary that $\sigma_k < \sigma_{k+1}$. 
By the strong Markov property, we then have 
\begin{align*}
\prcond{\tau_k < \sigma_{k+1}}{X_0^{(k)}\ne z_k}{} & = \pr{\sigma_k<\sigma_{k+1}}\prcond{\tau_k < \sigma_{k+1}}{X_0^{(k)}=z_k}{} \\
									& = \frac{q_k}{\sum_{j=k}^n q_j} \cdot \prcond{\tau_k < \sigma_{k+1}}{X_0^{(k)}=z_k}{}\, \\
									& \geq  \frac{1-\eps}{1-\eps^{n+1-k}}\pr{\tau_{k-1} < \sigma_k}\, . 
\end{align*}
We thus have the unconditional bound 
\[
\pr{\tau_k < \sigma_{k+1}} \ge \frac{1-\eps}{1-\eps^{n+1-k}} \pr{\tau_{k-1} < \sigma_k} > (1-\eps) \pr{\tau_{k-1} < \sigma_k}, 
\]
This bound holds for all $1 \le k < n$; since $\tau_0=0$ we also have $\pr{\tau_0< \sigma_1}=1$, and (\ref{eq:a_claim}) follows.

\section*{Acknowledgements}

The authors thank the Bellairs Institute, Microsoft Research and the Isaac Newton Institute, where parts of this work were completed.

\bibliographystyle{plain}
\bibliography{biblio}

\begin{thebibliography}{1}

\bibitem{AldFill}
David Aldous and James Fill.
\newblock {\em Reversible Markov Chains and Random Walks on Graphs}.
\newblock In preparation,
  http://www.stat.berkeley.edu/$\sim$aldous/RWG/book.html.

\bibitem{AldBrown}
David~J. Aldous and Mark Brown.
\newblock Inequalities for rare events in time-reversible {M}arkov chains. {I}.
\newblock In {\em Stochastic inequalities ({S}eattle, {WA}, 1991)}, volume~22
  of {\em IMS Lecture Notes Monogr. Ser.}, pages 1--16. Inst. Math. Statist.,
  Hayward, CA, 1992.

\bibitem{LevPerWil}
David~A. Levin, Yuval Peres, and Elizabeth~L. Wilmer.
\newblock {\em Markov chains and mixing times}.
\newblock American Mathematical Society, Providence, RI, 2009.
\newblock With a chapter by James G. Propp and David B. Wilson.

\bibitem{Oliveira}
Roberto~Imbuzeiro Oliveira.
\newblock On the coalescence time of reversible random walks.
\newblock {\em Trans. Amer. Math. Soc.}, 364(4):2109--2128, 2012.

\bibitem{SousiWinkler}
Perla Sousi and Peter Winkler.
\newblock Mixing times and moving targets.
\newblock {\em Combin. Probab. Comput.}, 23(3):460--476, 2014.

\end{thebibliography}

\end{document}